\documentclass[reqno,11pt]{amsart}
\usepackage{relsize}
\usepackage{scalerel}
\usepackage{stackengine,wasysym}
\usepackage{todonotes}
\usepackage{xcolor}
\usepackage{mathrsfs}
\usepackage{dsfont}
\usepackage{mathtools}
\usepackage[colorlinks,
            linkcolor=blue,
            anchorcolor=blue,
            citecolor=blue]{hyperref}
\usepackage[sort,nocompress]{cite}

\allowdisplaybreaks

\numberwithin{equation}{section}

\usepackage{amsmath,amssymb} 
\usepackage[toc,page]{appendix}
\usepackage{bbm}
\usepackage{verbatim}
\allowdisplaybreaks
\theoremstyle{plain}
\newtheorem{lemma}{Lemma}[section]
\newtheorem{theorem}[lemma]{Theorem}
\newtheorem{corollary}[lemma]{Corollary}

\newtheorem{proposition}[lemma]{Proposition}

\theoremstyle{remark}
\newtheorem{remark}{Remark}

\def\G{\mathcal{G}}

\title[Regularization by noise for the inviscid PE]{Regularization by noise for the inviscid primitive equations}
\author[R. Hu]{Ruimeng Hu}
\address[R. Hu]
{	Department of Mathematics \\
Department of Statistics and Applied Probability \\
     University of California  \\
	Santa Barbara, CA 93106, USA.} \email{rhu@ucsb.edu}

\author[Q. Lin]{Quyuan Lin}
\address[Q. Lin]
{	School of Mathematical and Statistical Sciences \\
Clemson University\\
Clemson, SC 29634, USA.} \email{quyuanl@clemson.edu}

 \author[R. Liu]{Rongchang Liu}
\address[R. Liu]{Department of Mathematics, University of Arizona, Tucson, AZ 85721}
\email{lrc666@math.arizona.edu}

\date{July 31, 2024}

\begin{document}

\begin{abstract}
    The deterministic inviscid primitive equations (also called the hydrostatic Euler equations) are known to be ill-posed in Sobolev spaces and in Gevrey classes of order strictly greater than 1, and some of their analytic solutions exist only locally in time and exhibit finite-time blowup. This work demonstrates that introducing suitable random noise can restore the local well-posedness and prevent finite-time blowups. Specifically, random diffusion addresses the ill-posedness in certain Gevrey classes, allowing us to establish the local well-posedness almost surely and the global existence of solutions with high probability. In the case of random damping (linear multiplicative noise), the noise prevents analytic solutions from forming singularities in finite time, resulting in globally existing solutions with high probability.
\end{abstract}

\maketitle



\section{Introduction}

We are interested in the following three-dimensional stochastic inviscid primitive equations (PE) 
\begin{subequations}\label{PE-system} 
\begin{align}
    &d V + (V\cdot \nabla_h V + w\partial_z V +\nabla_h p ) dt = F(V)dW_t  , \label{PE-1}  
    \\
    &\partial_z p  =0 , \label{PE-2}
    \\
    &\nabla_h \cdot V + \partial_z w =0,  \label{PE-3} 
\end{align}
\end{subequations}
which is also known as the stochastic hydrostatic Euler equations. Here, the horizontal velocity field $V=(u, v)$, the vertical velocity $w$, and the pressure $p$ are the unknown quantities. We employ the notations $\nabla_h = (\partial_x, \partial_y)$ and $\nabla = (\partial_x, \partial_y, \partial_z)$ to distinguish between the horizontal and complete gradients. The Wiener process is denoted by $W_t$, and $F(V)dW_t$ represents the driving noise, the form of which will be specified later. The domain of interest is the three-dimensional torus $\mathbb T^3 = \mathbb R^3/\mathbb Z^3$ with unit volume, and the boundary conditions are
\begin{equation}\label{BC-T3}
    V \text{ is periodic in }  (x,y,z) \text{ with period }  1,  \qquad V  \text{ is even in }  z  \text{ and }  w  \text{ is odd in }  z.
\end{equation} 
Observe that the space of periodic functions with respect to $z$ with the symmetry condition \eqref{BC-T3} is
invariant under the dynamics of system \eqref{PE-system}.

The viscous PE is derived from the Navier-Stokes equations \cite{azerad2001mathematical,li2019primitive,li2022primitive,furukawa2020rigorous}, while the inviscid PE is derived from the Euler equations \cite{brenier2003remarks,grenier1999derivation,masmoudi2012h} by taking the hydrostatic limit. This model finds extensive application when the vertical-to-horizontal scale ratio is small. For example, it is particularly relevant in the study of large-scale oceanic and atmospheric dynamics, where the vertical scale (a few kilometers in oceans and 10-20 kilometers in the atmosphere) is significantly less than the horizontal scale, which spans thousands of kilometers.

Our work focuses on the well-posedness of the stochastic PE. Inspired by \cite{buckmaster2020surface, rosenzweig2023global,glatt2014local},  we examine solutions to the PE perturbed by a specific type of noise, that is,  $F(V)=\nu|\nabla|^s V$, where $\nu>0$ and $0\leq s\leq 1$ are constants, and $|\nabla|= \sqrt{-\Delta}$.
We further distinguish two different cases: random diffusion in which $s\in (0,1]$ and random damping in which $s=0$. Our primary focus is on the regularization effects these noises have on the system. Specifically, random diffusion helps to overcome the ill-posedness in Gevrey class of order $\frac1s$, allowing for the local well-posedness almost surely and the 
global existence of solutions with high probability. In contrast, random damping prevents the finite-time blowup for analytic solutions with high probability. Following the spirit of previous works \cite{buckmaster2020surface, rosenzweig2023global,glatt2014local}, we interpret the equation \eqref{PE-system} with such a pseudo-differential noise as an equivalent random PDE (see \eqref{e.070201} below) through a transformation involving a geometric Brownian motion and the multiplier $\nu|\nabla|^s$. The main results of our paper are summarized below. For the precise definition of the Gevrey class, see \eqref{e.070801} in the next section.  
\begin{theorem}
    Consider system \eqref{PE-system} subject to random noise of the form $F(V)=\nu|\nabla|^s V$.   
    \begin{itemize}
        \item (Random diffusion, Theorem~\ref{theorem:global}.) Let $\eta, \alpha>0$, $s\in(\frac45,1]$ and $\sigma\in(\frac8{5s},2)$. With probability one, system \eqref{PE-system} is locally well-posed with solutions in the Gevrey class $\G_{\eta}^{\sigma,s}$ for each deterministic initial data in $\G_{\alpha+\eta}^{\sigma,s}$. Moreover, for any $\varepsilon>0$, there is a sufficiently large constant $\alpha=\alpha(\varepsilon)$ such that for any given deterministic initial data in $\G_{\alpha+\eta}^{\sigma,s}$, by choosing a sufficiently large diffusion intensity $\nu$, the corresponding local solutions can be extended to global ones in $\G_{\eta}^{\sigma,s}$ with probability at least $1-\varepsilon$. 
        \item (Random damping, Theorem~\ref{thm:damping}.) Let $s=0$. The system \eqref{PE-system} is locally well-posed in the analytic class with probability one. Furthermore, given any $\varepsilon>0$ and deterministic analytic initial data, the local solution can be made global with probability at least $1-\varepsilon$ by taking sufficiently large damping intensity $\nu$.
    \end{itemize}
\end{theorem}

The primitive equations have been widely studied in the literature. For instance, the PE is globally well-posed in 3D when the system has full viscosity \cite{cao2007global,kobelkov2006existence,kukavica2007regularity,hieber2016global} in the deterministic setting. For the stochastic PE with full viscosity, local and global well-posedness were examined in \cite{glatt2008stochastic,glatt2011pathwise,brzezniak2021well,debussche2011local,debussche2012global,agresti2022stochastic-1,agresti2022stochastic-2}, and the existence and regularity of invariant measures were established in \cite{glatt2014existence}. On the other hand, the deterministic inviscid PE are ill-posed in Sobolev spaces and Gevrey classes of order strictly greater than 1 \cite{renardy2009ill,han2016ill,ibrahim2021finite}, and some smooth solutions can form singularities in finite time \cite{cao2015finite,wong2015blowup,ibrahim2021finite,collot2023stable}. To obtain local well-posedness, one needs to assume certain special structures (local Rayleigh condition) on the initial data in $2D$ \cite{brenier1999homogeneous,brenier2003remarks,grenier1999derivation,masmoudi2012h}, or real analyticity for general initial data in both $2D$ and $3D$ \cite{ghoul2022effect,kukavica2011local,kukavica2014local}. In the stochastic setting, the local well-posedness of the inviscid PE under the same initial assumptions was established in \cite{hu2023local,hu2023pathwise}, where the authors considered general multiplicative noise but did not explore the benefits of certain types of noise. 

In \cite{buckmaster2020surface, rosenzweig2023global}, the authors considered adding random diffusion $F(V)=\nu|\nabla|^s V$ to the inviscid surface quasi-geostrophic equation and a large class of active scalar equations (such as Patlak–Keller–Segel model), and showed that these systems are globally well-posed in the Gevrey class of order $\frac1s$ with high probability. Inspired by these works, we investigate the effect of random diffusion in the inviscid PE and establish that the system not only overcomes the ill-posedness in the Gevrey class of order $\frac1s$, but also prevents the blowup with high probability, in contrast to the blowup results in the deterministic case \cite{cao2015finite,wong2015blowup,ibrahim2021finite,collot2023stable}. The special nonlinear structure of the PE, which involves a loss of horizontal derivative of $V$,  distinguishes it from those models considered in the previous works \cite{buckmaster2020surface, rosenzweig2023global}.  To our best knowledge, this is the first study to show that random diffusion can tame an ill-posed system, resulting in local well-posedness and global solutions with high probability. 

In \cite{glatt2014local}, the authors considered the $3D$ stochastic Euler equations with random damping (linear multiplicative noise) $F(V)=\nu V$ and proved that regular solutions exist globally with high probability $1-\varepsilon$ when the initial condition meets certain size requirements depending on $\nu$ and $\varepsilon$. In light of recent breakthroughs on the blowup in the $3D$ deterministic Euler equations \cite{elgindi2021finite,chen2022stable}, 
this demonstrates that random damping can help prevent the formation of finite-time singularities. This motivates us to consider random damping for the inviscid PE and prove the global existence of analytic solutions with high probability, provided that the damping intensity $\nu$ is large enough.

The rest of the paper is organized as follows. In Section \ref{sec:pre}, we introduce the notations, functional settings, and some preliminaries for the paper. Sections \ref{sec:diffusion} and \ref{sec:damping} are devoted to the study of the inviscid PE subject to random diffusion and random damping, respectively. 
Finally, we conclude in Section \ref{sec:conclusion}.

\section{Preliminaries}\label{sec:pre}

In this section, we introduce notations and some needed preliminary results. The universal constant $C$ appearing below may change from line to line. When needed, we use subscripts to indicate the dependence of the constant on certain parameters, e.g., we write $C_\sigma$ to emphasize that the constant depends on $\sigma$.

Let $x:= (x',z) = (x_1, x_2, z)\in \mathbb{T}^3$, where $x'$ and $z$ represent the horizontal and vertical variables, respectively. Define
$
    \|f\|_{L^2} := \left(\int_{\mathbb{T}^3} |f(x)|^2 dx\right)^{\frac{1}{2}},
$
associated with the inner product
$  \langle f,g\rangle = \int_{\mathbb{T}^3} f(x)g(x) dx
$
for $f,g \in L^2(\mathbb{T}^3)$. For a function $f \in L^2(\mathbb{T}^3)$, let $\hat{f}_{k}$ be its Fourier coefficient such that
\begin{equation*}
    f(x) = \sum\limits_{k\in 2\pi\mathbb{Z}^3} \hat{f}_{k} e^{ ik\cdot x}, \qquad \hat{f}_{k} = \int_{\mathbb{T}^3} e^{- ik\cdot x} f(x) dx.
\end{equation*}
The $L^2$ norm of $f$ can also be represented as 
\[
   \|f\|_{L^2} = \Big(\sum\limits_{k\in 2\pi \mathbb Z^3} |\hat f_k|^2\Big)^{1/2}.
\]
For $s> 0$ we define the following Sobolev $H^s$ norm and $\dot{H}^s$ semi-norm
\begin{eqnarray*}
&&\hskip-.1in
\|f \|_{H^s}:= \Big(\sum\limits_{k\in 2\pi\mathbb{Z}^3} (1+|k|^{2s}) |\hat{f}_{k}|^2  \Big)^{\frac{1}{2}}, \qquad \|f \|_{\dot{H}^s}:= \Big(\sum\limits_{k\in  2\pi\mathbb{Z}^3} |k|^{2s} |\hat{f}_{k}|^2  \Big)^{\frac{1}{2}},
\end{eqnarray*}
and we refer to \cite{adams2003sobolev} for more details about Sobolev spaces. When $f$ has zero mean, i.e., $\hat f_0=0$, these two norms are equivalent, and $\|f\|_{H^s} \leq \sqrt{2}\|f\|_{\dot H^s}$.

Let $A=|\nabla|$ be the Fourier multiplier such that $\widehat{(Af)}_k = |k|\hat f_k$. For $s\in(0,1]$, $\sigma>0$, and a positive function $\phi(t)$, we define the Gevrey class with order $\frac1s\geq 1$ and radius $\phi(t)>0$ as
\begin{align}\label{e.070801}
    \mathcal{G}_{\phi}^{\sigma,s}: = \{f\in H^{\sigma s}: \|f\|_{\G_{\phi}^{\sigma,s}}<\infty\},
\end{align}
where the Gevrey norm is given by 
\[
\|f\|_{\G_{\phi}^{\sigma,s}} := \left(\|e^{\phi A^{s}}f\|_{\dot H^{\sigma s}}^2 + \|f\|_{L^2}^2 \right)^{\frac12}= \Big( \sum\limits_{k\in  2\pi\mathbb{Z}^3} (1+ e^{2\phi(t)|k|^s} |k|^{2\sigma s}) |\hat{f}_{k}|^2\Big)^{1/2}.
\]
The norm above is equivalent to the homogeneous Gevrey norm 
\[
\|f\|_{\dot\G_{\phi}^{\sigma,s}} := \|e^{\phi A^{s}}f\|_{\dot H^{\sigma s}}= \Big( \sum\limits_{k\in  2\pi\mathbb{Z}^3} e^{2\phi(t)|k|^s} |k|^{2\sigma s} |\hat{f}_{k}|^2\Big)^{1/2},
\]
for $f$ having zero mean,
and it holds that $\|f\|_{\G_{\phi}^{\sigma,s}} \leq \sqrt2 \|f\|_{\dot\G_{\phi}^{\sigma,s}}$.
Note that we often write $\phi$ instead of $\phi(t)$ for simplicity. When $s=1$, the Gevrey class coincides with the space of analytic functions with radius $\phi$. 

For a given function $f$, we define its barotropic part $\bar{f}$ and baroclinic part $\tilde f$ as 
\begin{align*}
    \bar{f}(x') = \int_0^1 f(x',z) dz, \quad \tilde f(x) = f(x) - \bar{f}(x').
\end{align*}
Let $\mathcal P_h$ be the $2D$ Leray projection such that $\mathcal P_h \overline \varphi = \overline{\varphi} - \nabla_h  \Delta_h^{-1} \nabla_h \cdot \overline{\varphi}$. Here $ \Delta_h^{-1} $ represents the inverse of the Laplacian operator on $ \mathbb T^2 $. Furthermore, let $\mathcal P$ be the hydrostatic Leray projection such that $\mathcal P \varphi = \mathcal P_h \overline \varphi + \tilde \varphi$. As $e^{\phi A^s}$ and $A^s$ are Fourier multipliers, these operators commute with $\mathcal P$. Moreover, $\mathcal P$ is a bounded operator in $L^2$, and consequently, one has
\begin{equation}\label{Leray:bdd}
    \|\mathcal P f\|_{\G^{\sigma,s}_\phi} \leq \|f\|_{\G^{\sigma,s}_\phi}, \quad \|\mathcal P f\|_{\dot \G^{\sigma,s}_\phi} \leq \|f\|_{\dot \G^{\sigma,s}_\phi}.
\end{equation}

From the divergence free condition \eqref{PE-3} and the boundary condition \eqref{BC-T3}, the vertical velocity $w$ can be represented through $V$ as
\begin{equation*}
    w(V) = -\int_0^z \nabla_h \cdot V (x',\tilde z) d\tilde z.
\end{equation*}
In addition, we define the bilinear form $Q(U,V)$ as
\begin{equation*}
    Q(U,V) = U\cdot\nabla_h V + w(U)\partial_z V.
\end{equation*}
Using these notations and applying $\mathcal P$ to \eqref{PE-1}, system \eqref{PE-system} can be rewritten as
\begin{equation}\label{PE-abstract}
    dV + \mathcal{P}Q(V,V)dt  = \mathcal P F(V) dW_t, \qquad V(0)=V_0. 
\end{equation}

Since the noise has the form $F(V)=\nu |\nabla|^s V$, integrating \eqref{PE-abstract} over $\mathbb T^3$ yields $\int_{\mathbb T^3} V(x,t) dx =0$ for all $t\geq0$, provided that $\int_{\mathbb T^3} V_0 dx =0$. Additionally, noting that $\overline{V}$ satisfies $\nabla_h \cdot \overline{V}=0$, we search solutions $V$ in $H$ for all $t\geq 0$, where
\begin{align}
    H:= \left\{f\in L^2: \int_{\mathbb T^3} f(x,t) dx =0, \,\, \nabla_h\cdot \bar{f}=0 \right\}.
\end{align}
For later use, we also define $C_T\G_{\phi}^{\sigma,s}$ for any given positive time $T>0$:
\[C_T\G_{\phi}^{\sigma,s}: =\left\{f\in C([0,T];H\cap \G_{\phi}^{\sigma,s}): \|f\|_{C_{T}\G_{\phi}^{\sigma,s}} := \sup\limits_{t\in[0,T]} \|f(t)\|_{\G_{\phi(t)}^{\sigma,s}}<\infty\right\}.\]

To solve \eqref{PE-abstract}, we introduce the multiplier $\Gamma(t) = e^{-\nu W_t |\nabla|^s}$ as in \cite{buckmaster2020surface, rosenzweig2023global,glatt2014local} and define $U=\Gamma V$. By applying It\^o's formula to $U=\Gamma V$ and using system \eqref{PE-abstract}, we can derive the following random PDE for $U$: 
\begin{align}\label{e.070201}
    dU+\mathcal{P}\Gamma Q(\Gamma^{-1}U,\Gamma^{-1}U)dt= -\frac12 \nu^2|\nabla|^{2s}Udt, \qquad U(0)= U_0 = V_0. 
\end{align}
In the sequel, we will focus on this random equation and use its solution to construct the solution to \eqref{PE-abstract}. Specifically, once we obtain solutions $U$ to 
\eqref{e.070201}, $V=\Gamma^{-1}U$ will give the solutions to \eqref{PE-abstract}. 

Regarding the stochastic setting, we consider 
the canonical realization of $(\Omega,\mathcal{F},\mathbb{P})$ where $\Omega=\{\omega\in C([0,\infty);\mathbb{R}):\omega(0)=0\}$
is the canonical Wiener space endowed with the compact-open topology, $\mathcal{F}$ is the Borel $\sigma$-algebra, and $\mathbb{P}$ is the Wiener measure. The Brownian motion is the coordinate process $W_t(\omega)=\omega(t)$ so that each sample path is continuous in time. For global solutions, we will work on the following good subset of $\Omega$:
\begin{align}\label{def:samples}
    \Omega_{\alpha,\beta,\nu} = \{\omega\in\Omega: \alpha+\beta t-\nu W_t\geq 0 \text{ for all } t\geq 0\}. 
\end{align}
For $\alpha,\beta>0$, this set has probability $\mathbb{P}(\Omega_{\alpha,\beta,\nu})\geq 1-e^{-\frac{\alpha\beta}{\nu^2}}$ (cf. \cite{resnick2013adventures}). When applying this 
estimate below, we will be in a situation where $\beta$ has the same order as $\nu^2$, thus to make the probability high, we need sufficiently large $\alpha$. See also \cite{buckmaster2020surface,rosenzweig2023global} for relevant discussions.

\section{Random diffusion}\label{sec:diffusion}
In this section, we analyze the stochastic inviscid PE subject to random diffusion, where $F(V)=\nu|\nabla|^sV$ in system \eqref{PE-abstract}. This is a special type of multiplicative noise defined through the Fourier multiplier, first considered in \cite{buckmaster2020surface}. The equation then reads 
\begin{align}\label{eq.L061001}
    dV + \mathcal{P}Q(V,V)  dt = \nu|\nabla|^sVdW_t, \qquad V(0)=V_0,
\end{align}
where the differentiability index $0<s\leq 1$ and the diffusivity $\nu>0$ are constants. 

Recall the definition of $\Gamma(t) = e^{-\nu W_t|\nabla|^s}$ and $U=\Gamma V$. It\^o's formula gives that 
\[d\Gamma = -\nu|\nabla|^s\Gamma dW_t + \frac{1}{2}\nu^2|\nabla|^{2s}\Gamma dt,\]
and $U$ solves the random PDE:
\begin{align}\label{eq.L061002}
    \partial_t U+\mathcal{P}\Gamma Q(\Gamma^{-1}U,\Gamma^{-1}U)= -\frac12 \nu^2|\nabla|^{2s}U, \qquad U(0)= U_0 = V_0,
\end{align}
provided that $V$ is a solution to \eqref{eq.L061001}, where we have used $\Gamma(0)=1$ to obtain $U_0 = V_0$. As mentioned above, we shall primarily work with this random PDE and use its solution $U$ to construct the solution $V$ to \eqref{eq.L061001} via $V = \Gamma^{-1}U$.

In the rest of this section, we set $\phi(t) = \alpha+\beta t$ as the Gevrey radius, where  $\alpha>0$ and  $0<\beta<\frac{\nu^2}{2}$.

\subsection{Local existence}
We first provide the local existence of solutions to \eqref{eq.L061002}. Compared to \cite{buckmaster2020surface, rosenzweig2023global}, the special nonlinear structure of the PE requires a delicate analysis to obtain the optimal range of parameters $\sigma$ and $s$.

\begin{proposition}\label{p.061701}
      Let $s\in(\frac45,1]$ and $\sigma\in(\frac8{5s},2)$. For any initial data $U_0\in H\cap \G_{\alpha}^{\sigma,s}$, there exists a random time $T>0$,  such that there is a unique local solution to \eqref{eq.L061002} in $C_T\G_{\phi}^{\sigma,s}$ for every $\omega\in \Omega$. Moreover, the random existence time $T$ becomes deterministic over $\Omega_{\alpha,\beta,\nu}$ and depends only on the size of the initial data. 
\end{proposition}
\begin{proof}
    For any fixed $\omega\in \Omega$, define 
    \[T_{\omega} :=\inf\left\{t\geq 0: \nu W_t(\omega)>\phi(t)\right\}.\]
    Note that $T_{\omega}>0$ by the continuity of the Brownian path $W_t(\omega)$ and the fact that $W_0(\omega)=0$ and $\phi(0)=\alpha>0$. Denote by \begin{align}\label{e.L061101}
        B(U,U) = \Gamma Q(\Gamma^{-1}U,\Gamma^{-1}U).
    \end{align}
    Since $\Gamma$ commutes with $\mathcal P$, from \eqref{eq.L061002} we define the solution operator 
    \begin{align*}
        \Phi(U)(t):= e^{-\frac{\nu^2t}{2}A^{2s}}U_0 - \int_0^t e^{-\frac{\nu^2(t-r)}{2}A^{2s}} \mathcal{P}B(U,U) dr
    \end{align*} 
    for $t\in[0,T_{\omega}]$. Note that since both $U_0$ and $B(U,U)$ have zero means, so does $\Phi(U)$.
    
    We want to estimate the Gevrey norm of $\Phi(U)$. 
    For any $j=(j_1,j_2,j_3)\in \mathbb Z^3$ we denote by $j'=(j_1,j_2)$ and similar notations apply to $k = (k_1, k_2, k_3)$. By our assumption $\beta<\frac{\nu^2}{2}$ and the property \eqref{Leray:bdd} of the hydrostatic Leray projection, we have
    \begin{align*}
        \|\Phi(U)(t)\|_{\G^{\sigma,s}_{\phi(t)}} \leq & \sqrt{2} \|\Phi(U)(t)\|_{\dot\G^{\sigma,s}_{\phi(t)}} = \sqrt{2}
        \|e^{\phi(t) A^s}\Phi(U)(t)\|_{\dot H^{\sigma s}}
        \\
        \leq &\sqrt2\|e^{\phi(t) A^s-\frac{\nu^2t}{2}A^{2s}}U_0\|_{\dot H^{\sigma s}} + \sqrt{2}\int_0^t\left\|e^{\phi(t)A^s-\frac{\nu^2}{2}(t-r)A^{2s}}B(U,U)\right\|_{\dot H^{\sigma s}}dr \\
        \leq &\sqrt{2}\|e^{\alpha A^s}U_0\|_{\dot H^{\sigma s}} + \sqrt{2}\int_0^t\left\|e^{\phi(t)A^s-\frac{\nu^2}{2}(t-r)A^{2s}}B(U,U)\right\|_{\dot H^{\sigma s}}dr.
    \end{align*}
    We next estimate the nonlinear term. A straightforward computation yields 
    \begin{align}\label{e.L061102}
    \begin{split}
        \widehat{B(U,U)}(t,k)&= e^{-\nu W_t|k|^s}\sum_{j}e^{\nu W_t|j|^s}\hat{U}_j\cdot i(k'-j')e^{\nu W_t|k-j|^s}\hat{U}_{k-j}\\
        &\quad -e^{-\nu W_t|k|^s}\sum_{j}e^{\nu W_t|j|^s}\frac{j'\cdot\hat{U}_j}{j_3}i(k_3-j_3)e^{\nu W_t|k-j|^s}\hat{U}_{k-j}\\
        &:=\hat{B}_1(t,k)+\hat{B}_2(t,k). 
    \end{split}
    \end{align}
    
    For the term $\hat{B}_2(t,k)$, using the facts that $\phi(t)-\nu W_t(\omega)\geq 0$ for $t\in[0,T_{\omega}]$, and $|k|^s\leq |k-j|^s+|j|^s$ when $0<s\leq 1$, one has
    \begin{align*}
        &\sum_{k}\left||k|^{\sigma s}e^{\phi(t)|k|^s-\frac{\nu^2}{2}(t-r)|k|^{2s}}\hat{B}_2(r, k)\right|^2\\
        & = \sum_{k}\left||k|^{\sigma s}e^{(\phi(t)-\phi(r))|k|^s-\frac{\nu^2}{2}(t-r)|k|^{2s}}\sum_{j}e^{(\phi(r)-\nu W_r)(|k|^s-|j|^s-|k-j|^s)}e^{\phi(r)|j|^s}\frac{j'\cdot\hat{U}_j}{j_3}(k_3-j_3)e^{\phi(r)|k-j|^s}\hat{U}_{k-j}\right|^2\\
        &\leq \sum_{k}\left||k|^{\sigma s}e^{(\beta |k|^s-\frac{\nu^2}{2}|k|^{2s})(t-r)}\sum_{j}e^{(\phi(r)-\nu W_r)(|k|^s-|j|^s-|k-j|^s)}e^{\phi(r)|j|^s}\frac{j'\cdot\hat{U}_j}{j_3}(k_3-j_3)e^{\phi(r)|k-j|^s}\hat{U}_{k-j}\right|^2\\
        &\leq \sum_{k}|k|^{2\sigma s}e^{2(\beta |k|^s-\frac{\nu^2}{2}|k|^{2s})(t-r)}\Bigg(\sum_{j}\left|e^{\phi(r)|j|^s}\frac{j'\cdot\hat{U}_j}{j_3}(k_3-j_3)e^{\phi(r)|k-j|^s}\hat{U}_{k-j}\right|\Bigg)^2.
    \end{align*}
    Since $\beta<\frac{\nu^2}{2}$, there exist a constant $C>0$ such that  
    $|k|^{2\sigma s}e^{2(\beta |k|^s-\frac{\nu^2}{2}|k|^{2s})(t-r)}\leq C(t-r)^{-\sigma}$.
    By Young's inequality for convolution with $\frac{1}{p}+\frac{1}{q}=\frac32$ and $1< p,q <2$, one has
    \begin{align*}
        &\Bigg(\sum_{k}\Bigg(\sum_{j}\left|e^{\phi(r)|j|^s}\frac{j'\cdot\hat{U}_j}{j_3}(k_3-j_3)e^{\phi(r)|k-j|^s}\hat{U}_{k-j}\right|\Bigg)^2\Bigg)^{\frac12}\\
        &\leq \Bigg(\sum_{j}\left|e^{\phi(r)|j|^s}\frac{|j'|}{|j_3|}\hat{U}_j\right|^p\Bigg)^{1/p}\Bigg(\sum_{j}\left||j|e^{\phi(r)|j|^s}\hat{U}_{j}\right|^q\Bigg)^{1/q}.
    \end{align*}
    Using the H\"older inequality, the first term in the above inequality is bounded by
    \begin{align}\label{e.061701}
    \begin{split}
       &\Bigg(\sum_{j}\left|e^{\phi(r)|j|^s}\frac{|j'|}{|j_3|}\hat{U}_j\right|^p\Bigg)^{1/p} = \Bigg(\sum_{j_3\neq 0} \sum_{j'\neq 0}\left|e^{\phi(r)|j|^s}\frac{|j'|}{|j_3|}\hat{U}_j\right|^p\Bigg)^{1/p}
       \\
       \leq & \Bigg(\sum_{j_3\neq 0} \Bigg(\sum_{j'\neq 0}\left|e^{p\phi(r)|j|^s}|j|^{p\sigma s}|\hat{U}_j|^p\right|^{\frac2p}\Bigg)^{p/2} \Bigg(\sum_{j'\neq 0} \left| |j_3|^{-p} |j|^{p-p\sigma s} \right|^{\frac{2}{2-p}}  \Bigg)^{\frac{2-p}2}  \Bigg)^{1/p}
       \\
       \leq & \Bigg(\sum_{j_3\neq 0} \Bigg(\sum_{j'\neq 0}e^{2\phi(r)|j|^s}|j|^{2\sigma s}|\hat{U}_j|^2\Bigg)^{p/2} |j_3|^{-p} \Bigg(\sum_{j'\neq 0} |j|^{\frac{2p(1-\sigma s)}{2-p}}  \Bigg)^{\frac{2-p}2}  \Bigg)^{1/p}
       \\
       \leq & C \Bigg(\sum_{j_3\neq 0} \Bigg(\sum_{j'\neq 0}e^{2\phi(r)|j|^s}|j|^{2\sigma s}|\hat{U}_j|^2\Bigg)^{p/2} |j_3|^{-p}   \Bigg)^{1/p}
       \\
       \leq & C \Bigg(\sum_{j_3\neq 0}  \sum_{j'\neq 0}e^{2\phi(r)|j|^s}|j|^{2\sigma s}|\hat{U}_j|^2\Bigg)^{1/2} \Bigg(\sum_{j_3\neq 0}  |j_3|^{-\frac{2p}{2-p}} \Bigg)^{\frac{2-p}{2p}} \leq C \|e^{\phi(r)A^s}U\|_{\dot H^{\sigma s}},
    \end{split}
    \end{align}
    where in the third inequality we need the assumption that
    \begin{equation}\label{condition:p}
        \frac{2p}{2-p}(\sigma s-1) >2,
    \end{equation}
    and we have used the fact that $|j|^{-\frac{2p}{2-p}(\sigma s-1)} \leq |j'|^{-\frac{2p}{2-p}(\sigma s-1)}$, for $j_3\neq 0$.
    Similarly, for the second term we deduce
    \begin{align}\label{eq.Q062901}
    \left(\sum_{j}\left||j|e^{\phi(r)|j|^s}\hat{U}_{j}\right|^q\right)^{1/q} \leq &\left(\sum_{j}e^{2\phi(r)|j|^s}|j|^{2\sigma s}|\hat{U}_j|^2\right)^{1/2}\left(\sum_{j}|j|^{\frac{2q}{2-q}(1-\sigma s)}\right)^{\frac{2-q}{2q}}\notag
         \\
         \leq &C\|e^{\phi(r)A^s}U\|_{\dot H^{\sigma s}},
    \end{align}
    as long as assuming
    \begin{align}\label{eq.L061103}
        \frac{2q}{2-q}(\sigma s-1)>3.
    \end{align}
    
    To ensure that both \eqref{condition:p} and \eqref{eq.L061103} hold, it is equivalent to guarantee that
     $\sigma s-1> f(p,q)$, where 
    $f(p,q)=\max\left\{\frac{2-p}{p},\frac{3(2-q)}{2q}\right\}$.
    Since $1/p+1/q=3/2$, we consider the function
    \[g(p)=f(p,1/(3/2-1/p))=\max\left\{\frac2p-1,3(1-\frac1p)\right\},\]
    which has a minimum value of $3/5$ at $p=5/4$.
    Thus, as long as $\sigma s>8/5$, both \eqref{condition:p} and \eqref{eq.L061103} will hold. 
    Therefore, we have 
    \begin{align*}
        \left(\sum_{k}\left||k|^{\sigma s}e^{\phi(t)|k|^s-\frac{\nu^2}{2}(t-r)|k|^{2s}}\hat{B}_2(r, k)\right|^2\right)^{\frac12}\leq C\frac{\|e^{\phi(r)A^s}U\|_{\dot H^{\sigma s}}^2}{(t-r)^{\frac{\sigma}{2}}}. 
    \end{align*}
    
    Similarly, the same bound holds for the corresponding term involving $\hat{B}_1$, with the condition \eqref{condition:p} replaced by $\frac{2p}{2-p}\sigma s >3$, which automatically holds as long as \eqref{condition:p} holds since $1< p < 2$. 
    
    Therefore we have 
    \begin{align*}
        \int_0^t\|e^{\phi(t)A^s-\frac{\nu^2}{2}(t-r)A^{2s}}B(U,U)\|_{\dot H^{\sigma s}}dr
        &\leq C\int_0^t\frac{\|e^{\phi(r)A^s}U\|_{\dot H^{\sigma s}}^2}{(t-r)^{\frac{\sigma}{2}}} dr\\
        & \leq C\int_0^t\frac{1}{(t-r)^{\frac{\sigma}{2}}} dr\sup_{r\in[0,t]}\|e^{\phi(r)A^s}U\|_{\dot H^{\sigma s}}^2\\
        &\leq Ct^{1-\frac{\sigma}{2}}\sup_{r\in[0,t]}\|e^{\phi(r)A^s}U\|_{\dot H^{\sigma s}}^2,
    \end{align*}
    which requires $\sigma<2$. Thus, the valid ranges of $\sigma$ and $s$ are $s\in(\frac45,1]$ and $\sigma\in(\frac8{5s},2)$.
    Consequently, for any $T\in[0,T_{\omega}]$, we obtain 
    \[
    \|\Phi(U)\|_{C_T\G_{\phi}^{\sigma, s}}\leq \sqrt{2}\|e^{\alpha A^s}U_0\|_{\dot H^{\sigma s}}+ \sqrt{2}CT^{1-\frac{\sigma}{2}}\|U\|_{C_T\G_{\phi}^{\sigma, s}}^2. 
    \]
    
    Choose a radius $R\geq2\sqrt2\|e^{\alpha A^s}U_0\|_{\dot H^{\sigma s}}$ and denote by $B_R(0)$ the ball in $C_T\G_{\phi}^{\sigma, s}$ centered at the origin. Then for  $0<T\leq T_{\omega}$ satisfying $\sqrt2CT^{1-\frac{\sigma}{2}}R  \leq \frac1{2}$, the operator $\Phi$ maps $B_R(0)$ to itself.  
    Since the nonlinear term $B(U,U)$ is bilinear, we can use similar estimates as above to obtain
    \begin{align*}
        \|\Phi(U)-\Phi(V)\|_{C_T\G_{\phi}^{\sigma, s}}&\leq  CT^{1-\frac{\sigma}{2}}(\|U\|_{C_T\G_{\phi}^{\sigma, s}}+\|V\|_{C_T\G_{\phi}^{\sigma, s}})\|U-V\|_{C_T\G_{\phi}^{\sigma, s}}
        \\
        &\leq 2CT^{1-\frac{\sigma}{2}}R\|U-V\|_{C_T\G_{\phi}^{\sigma, s}}
    \end{align*}
    for all $U,V\in B_R(0)$ with $U(0)=V(0)$. As long as $T$ is small enough such that $2CT^{1-\frac{\sigma}{2}}R\leq \frac12$, $\Phi$ is a contraction on $B_R(0)$, and its unique fixed point corresponds to the unique local solution of \eqref{eq.L061002} for the sample $\omega$. The local solution of \eqref{eq.L061002} starting from $U_0$ is then obtained for every sample $\omega$ in $\Omega$ by the arbitrariness of $\omega\in \Omega$. 

Recall the definition of $\Omega_{\alpha,\beta,\nu}$ from \eqref{def:samples}. On $\Omega_{\alpha,\beta,\nu}$, one has $\phi(t)-\nu W_t\geq 0$ for all $t\geq 0$. Thus $T_\omega=\infty$ on $\Omega_{\alpha,\beta,\nu}$. Consequently, the random existence time $T$ can be made uniform for all $\omega\in \Omega_{\alpha,\beta,\nu}$, which means that the existence time $T$ becomes deterministic over $\Omega_{\alpha,\beta,\nu}$. 
\end{proof}

Transforming back to the original system \eqref{eq.L061001} using $V = \Gamma^{-1}U$, we obtain the local well-posedness in the Gevrey class of order $\frac1s$ for the stochastic inviscid PE subject to random diffusion. The result is summarized as follows. 
\begin{corollary}\label{thm:local}
    Let $\eta>0$, $s\in(\frac45,1]$ and $\sigma\in(\frac8{5s},2)$. For any $V_0\in H\cap \G_{\alpha+\eta}^{\sigma,s}$, with probability one, there exists a random time $T>0$ such that \eqref{eq.L061001} has a unique solution $V\in C_T\G_{\eta}^{\sigma,s}$. 
\end{corollary}
\begin{proof}
    Recall that $U(t)=e^{-\nu W_t A^s} V(t)$. Since $U_0=V_0$, we have $U_0\in \G_{\alpha+\eta}^{\sigma,s}$. By Proposition \ref{p.061701} we know that with probability one, the solution $U$ to equation \eqref{eq.L061002} exists in $C_T\G_{\phi+\eta}^{\sigma,s}$ up to a random time $T$. Since $\nu W_t-\phi(t)\leq 0$ for all $t\in[0,T]$, we have
    \[\|V(t)\|_{\G_{\eta}^{\sigma,s}}\leq \sqrt2\|e^{\eta A^s}e^{\nu W_t A^s}U(t)\|_{\dot H^{\sigma s}}= \sqrt2\|e^{(\eta+\phi(t)) A^s}e^{(\nu W_t-\phi(t)) A^s}U\|_{\dot H^{\sigma s}}\leq \sqrt2\|U\|_{\G_{\phi+\eta}^{\sigma,s}}.\]
    Therefore, $V\in C_T\G_{\eta}^{\sigma,s}$. 
\end{proof}

\subsection{Global existence}
We next investigate the global solutions to equation \eqref{eq.L061002} on the sample set $\Omega_{\alpha,\beta,\nu}$.  To this end, we first derive  energy estimates that allows us to extend the local solutions. 
\begin{proposition}\label{p.061702}
    Let $s\in(\frac45,1]$ and $\sigma\in(\frac8{5s},2)$. For $T>0$, assume that system \eqref{eq.L061002} has a unique solution $U\in C_T\G_{\phi}^{\sigma+1,s}$ for all sample paths in $\Omega_{\alpha,\beta,\nu}$. Then $U$ satisfies
    \begin{align*}
        \frac{d}{dt}\|U\|_{\G_\phi^{\sigma,s}}^2+(\nu^2-2\beta-C^*\|U\|_{\G_\phi^{\sigma,s}})\|U\|_{\dot\G_\phi^{\sigma+1,s}}^2\leq 0,
    \end{align*}
    where the constant $C^*>0$ depends only on $\sigma$ and $s$. In particular, if $\|U_0\|_{\G_{\phi(0)}^{\sigma,s}}\leq \frac{\nu^2-2\beta}{C^*}$, then $\|U(t)\|_{\G_{\phi(t)}^{\sigma,s}}$ is decreasing in time $t$ and
    $
        \|U(t)\|_{\G_{\phi(t)}^{\sigma,s}} \leq \|U_0\|_{\G_{\phi(0)}^{\sigma,s}}
    $
    for all $t\in[0,T]$.
\end{proposition}
\begin{proof}
    From equation \eqref{eq.L061002} and notation \eqref{e.L061101}, one has
    \begin{align*}
        \frac12 \frac{d}{dt} \|U\|_{L^2}^2 = - \frac{\nu^2}{2} \|A^s U\|_{L^2}^2
    \end{align*}
    due to the cancellation $\langle B(U,U), U\rangle=0$, and
    \begin{align}
        \frac12 \frac{d}{dt}\|e^{\phi(t)A^s}U\|_{\dot H^{\sigma s}}^2 = &\beta\|e^{\phi(t)A^s} A^{\frac s2}U\|_{\dot H^{\sigma s}}^2 - \frac{\nu^2}{2}\|e^{\phi(t)A^s} A^s U\|_{\dot H^{\sigma s}}^2 \notag
        \\
        &-\langle e^{\phi(t)A^s}B(U,U),e^{\phi(t)A^s}A^{2\sigma s}U\rangle.
    \end{align}
    Combining the above two equations together gives
    \begin{align}\label{e.061703}
        \frac12 \frac{d}{dt}\|U\|_{\G_{\phi}^{\sigma,s}}^2 = \beta\|U\|_{\dot\G_{\phi}^{\sigma+\frac12,s}}^2 - \frac{\nu^2}{2}\|A^s U\|_{\G_{\phi}^{\sigma,s}}^2 
        -\langle e^{\phi(t)A^s}B(U,U),e^{\phi(t)A^s}A^{2\sigma s}U\rangle.
    \end{align}
    In view of \eqref{e.L061102}, one has 
    \begin{align}\label{e.061702}
            &|\langle e^{\phi(t)A^s}B(U,U),e^{\phi(t)A^s}A^{2\sigma s}U\rangle|\\
            &\leq \sum_k\left|e^{\phi(t)|k|^s}\hat{B}_1(t,k)e^{\phi(t)|k|^s}|k|^{2\sigma s}\overline{\hat{U}}_k\right|+\sum_k\left|e^{\phi(t)|k|^s}\hat{B}_2(t,k)e^{\phi(t)|k|^s}|k|^{2\sigma s}\overline{\hat{U}}_k\right|:=B_1+B_2. \notag
    \end{align}
    Using the fact that  $\phi(t)-\nu W_t\geq 0$ for all $t\geq 0$ and $\omega \in \Omega_{\alpha,\beta,\nu}$, and $|k|^s\leq |k-j|^s+|j|^s$ for $0<s\leq 1$, one deduces, by the Cauchy–Schwarz inequality, that  
    \begin{align*}
        \begin{split}
            B_2 &= \sum_k\left|e^{(\phi(t)-\nu W_t)|k|^s}\sum_{j}e^{\nu W_t|j|^s}\frac{j'\cdot\hat{U}_j}{j_3}i(k_3-j_3)e^{\nu W_t|k-j|^s}\hat{U}_{k-j}e^{\phi(t)|k|^s}|k|^{2\sigma s}\overline{\hat{U}}_k\right|\\
            & = \sum_k\left|\sum_{j}e^{(\phi(t)-\nu W_t)(|k|^s-|k-j|^s-|j|^s)}e^{ \phi(t)|j|^s}\frac{j'\cdot\hat{U}_j}{j_3}i(k_3-j_3)e^{\phi(t)|k-j|^s}\hat{U}_{k-j}e^{\phi(t)|k|^s}|k|^{2\sigma s}\overline{\hat{U}}_k\right|\\
            &\leq \sum_k\sum_{j}\left|e^{ \phi(t)|j|^s}|k|^{(\sigma-1)s}\frac{j'\cdot\hat{U}_j}{j_3}i(k_3-j_3)e^{\phi(t)|k-j|^s}\hat{U}_{k-j}e^{\phi(t)|k|^s}|k|^{(\sigma+1)s}\overline{\hat{U}}_k\right|\\
            &\leq \|e^{\phi(t)A^s}U\|_{\dot H^{(\sigma+1)s}}\Bigg(\sum_k\Bigg(\sum_{j}\left|e^{ \phi(t)|j|^s}|k|^{(\sigma-1)s}\frac{j'\cdot\hat{U}_j}{j_3}i(k_3-j_3)e^{\phi(t)|k-j|^s}\hat{U}_{k-j}\right|\Bigg)^2\Bigg)^{\frac12}.
        \end{split}
    \end{align*}
    Using $|k|^{(\sigma-1)s}\leq C( |k-j|^{(\sigma-1)s}+|j|^{(\sigma-1)s})$ and the Minkowski inequality, the last term in the above inequality is controlled by 
    \begin{align*}
        \begin{split}
            &\Bigg(\sum_k\Bigg(\sum_{j}\left|e^{\phi(t)|j|^s}|k|^{(\sigma-1)s}\frac{j'\cdot\hat{U}_j}{j_3}i(k_3-j_3)e^{\phi(t)|k-j|^s}\hat{U}_{k-j}\right|\Bigg)^2\Bigg)^{\frac12}\\
            &\qquad \leq C\Bigg(\sum_k\Bigg(\sum_{j}\left|e^{\phi(t)|j|^s}|j|^{(\sigma-1)s}\frac{j'\cdot\hat{U}_j}{j_3}(k_3-j_3)e^{\phi(t)|k-j|^s}\hat{U}_{k-j}\right|\Bigg)^2\Bigg)^{\frac12}\\
            &\qquad \qquad+C\Bigg(\sum_k\Bigg(\sum_{j}\left|e^{\phi(t)|j|^s}\frac{j'\cdot\hat{U}_j}{j_3}|k-j|^{(\sigma-1)s}(k_3-j_3)e^{\phi(t)|k-j|^s}\hat{U}_{k-j}\right|\Bigg)^2\Bigg)^{\frac12}\\
            &\qquad :=B_{21}+B_{22}.
        \end{split}
    \end{align*}
    We first estimate the term $B_{21}$. By Young's inequality for convolution with $\frac1p+\frac1q=\frac32$ and $1< p,q< 2$, one has
    \begin{align*}
        \begin{split}
            B_{21}&\leq C\Bigg(\sum_{j_3\neq 0}\sum_{j'\neq 0}\left|e^{\phi(t)|j|^s}|j|^{(\sigma-1)s}\frac{|j'|}{|j_3|}\hat{U}_j\right|^p\Bigg)^{1/p}\Bigg(\sum_{j}\left||j|e^{\phi(t)|j|^s}\hat{U}_{j}\right|^q\Bigg)^{1/q}.
        \end{split}
    \end{align*}
Similar to the derivations of estimates \eqref{e.061701} and \eqref{eq.Q062901},  H\"older inequality gives 
\begin{align*}
    &\Bigg(\sum_{j'\neq 0}\left|e^{\phi(t)|j|^s}|j|^{(\sigma-1)s}\frac{|j'|}{|j_3|}\hat{U}_j\right|^p\Bigg)^{1/p}
    \\
     & \qquad \leq   C\Bigg(\sum_{j_3\neq 0} \Bigg(\sum_{j'\neq 0}\left|e^{p\phi(r)|j|^s}|j|^{p(\sigma+1) s}|\hat{U}_j|^p\right|^{\frac2p}\Bigg)^{\frac{p}{2}} \Bigg(\sum_{j'\neq 0} \left| |j_3|^{-p} |j|^{p(1-2s)} \right|^{\frac{2}{2-p}}  \Bigg)^{\frac{2-p}2}  \Bigg)^{\frac1p}
    \\
        & \qquad \leq \Bigg(\sum_{j_3\neq 0} \Bigg(\sum_{j'\neq 0}e^{2\phi(r)|j|^s}|j|^{2(\sigma+1) s}|\hat{U}_j|^2\Bigg)^{p/2} |j_3|^{-p} \Bigg(\sum_{j'\neq 0} |j'|^{\frac{2p(1-2 s)}{2-p}}  \Bigg)^{\frac{2-p}2}  \Bigg)^{1/p}
        \\
        &\qquad \leq C \|e^{\phi(r)A^s}U\|_{\dot H^{(\sigma+1) s}},
\end{align*}
and
 \begin{align*}
     \Bigg(\sum_{j}\left||j|e^{\phi(r)|j|^s}\hat{U}_{j}\right|^q\Bigg)^{1/q} \leq C\|e^{\phi(r)A^s}U\|_{\dot H^{\sigma s}},
 \end{align*}
    provided that
    \begin{align}\label{global-condition-1}
        \frac{2p}{2-p}(2 s-1)>2, \text{ and } \frac{2q}{2-q}(\sigma s-1)>3.
    \end{align}
    The same bound for $B_{22}$ can be obtained similarly, provided there exist $1< p,q<2$ with $\frac1p+\frac1q=\frac32$ such that
    \begin{align}\label{global-condition-2}
        \frac{2p}{2-p}(\sigma s-1)>2, \text{ and }\frac{2q}{2-q}(2s-1)>3.
    \end{align}
    Note that the above two conditions, \eqref{global-condition-1} and \eqref{global-condition-2}, are implied by \eqref{condition:p} and \eqref{eq.L061103} since $\sigma<2$. As a result, the assumptions in Proposition \ref{p.061702} on $\sigma$ and $s$ ensure the existence of such $p,q$, making the above inequalities hold. Thus we obtain 
    \[
    B_2\leq \|e^{\phi(t)A^s}U\|_{\dot H^{(\sigma+1)s}}(B_{21}+B_{22})\leq C\|U\|_{\dot \G_{\phi}^{\sigma,s}}\|U\|_{\dot \G_{\phi}^{\sigma+1,s}}^2,
    \]
    where $C$ depends only on $\sigma$ and $s$.
    The same bound also holds for the term $B_1$ defined in \eqref{e.061702} under the same conditions. Indeed, the estimate for $B_1$ is simpler than that for $B_2$ as there is no loss of horizontal derivative. 

    The proof is then complete by combining the bounds for $B_1$ and $B_2$ with \eqref{e.061703} and \eqref{e.061702}. 
\end{proof}

We are now ready to extend the local solutions globally. Note that the fixed point argument provides solutions in $C_{T}\G_{\phi}^{\sigma,s}$, but we require solutions to be in $C_{T}\G_{\phi}^{\sigma+1}$ when deriving the energy bound. To extend the local solutions, we apply a trick from \cite{buckmaster2020surface}, taking advantage of Gevrey embeddings. 
\begin{proposition}\label{thm:main}
    Let $s\in(\frac45,1]$ and $\sigma\in(\frac{8}{5s},2)$, and fix $\eta>0$. For $\alpha>0$ and $0<\beta<\frac{\nu^2}{2}$, assume that the initial condition $U_0\in H\cap \G_{\alpha+\eta}^{\sigma,s}$ satisfies
    \begin{align}\label{e.062001}
        \left\|U_0\right\|_{\mathcal{G}_{\alpha+\eta}^{\sigma,s}}<\frac1{C^*}\left(\frac{\nu^2}{2}-\beta\right),
    \end{align}
    where $C^*$ is the constant appearing in Proposition \ref{p.061702}. Then, system \eqref{eq.L061002} has a unique global solution in $C\left([0, \infty) ; \mathcal{G}_{\phi+\eta}^{\sigma,s}\right)$ for every sample $\omega \in \Omega_{\alpha, \beta, \nu}$, where $\phi(t)=\alpha+\beta t$. Furthermore, 
    $t\to \|U(t)\|_{\mathcal{G}_{\phi(t)+\eta}^{\sigma,s}}$
    is decreasing for all $t\geq 0$.  
\end{proposition}
\begin{proof}
    For a given constant $\eta>0$, we fix some $\eta'\in [0,\eta)$. We will first show by induction that the solution exists globally in $C([0,\infty);\G_{\phi+\eta'}^{\sigma,s})$. To this end, let us define  
    \[\phi_0=\phi+\eta \text{ and } \phi_{j}=\phi_{j-1}-2^{-j}(\eta-\eta') \text{ for } j\geq 1.\]
    Since $U_0\in {\mathcal{G}_{\alpha+\eta}^{\sigma,s}}$, by Proposition \ref{p.061701}, there exists a deterministic $T>0$, depending only on $\|U_0\|_{\G_{\alpha+\eta}^{\sigma,s}}$ and $\sigma$, such that system \eqref{eq.L061002} has a unique solution $U\in C\left([0, T]; \mathcal{G}_{\phi_0}^{\sigma,s}\right)$ for every sample  $\omega \in \Omega_{\alpha, \beta, \nu}$. 
    Our goal is to show that for all $j\geq 1$, the unique solution $U$ satisfies 
    \begin{align}\label{eq:induction}
        &U \in C\left([0, jT]; \G_{\phi_j}^{\sigma,s}\right), \qquad  t\to \|U(t)\|_{\G_{\phi_j(t)}^{\sigma,s}} \text{ is decreasing,} \notag\\
        &\hspace{30pt}\text{and }\|U(t)\|_{\G_{\phi_j(t)}^{\sigma,s}} \leq \|U_0\|_{\G_{\alpha+\eta}^{\sigma,s}} \text{ for } t\in[0,jT].
    \end{align}
    
    Let us start with $j=1$.
    Since $\phi_0(t)-\phi_1(t)=\frac12(\eta-\eta')=:\triangle \phi_0>0$ is independent of time $t$, it follows that
    \begin{align}
            &\|e^{\phi_1(t_1) A^s} U(t_1) - e^{\phi_1(t_2)} U(t_2) \|_{\dot H^{(\sigma+1)s}}^2 = \|e^{-\triangle \phi_0 A^s} \left(e^{\phi_0(t_1) A^s} U(t_1) - e^{\phi_0(t_2)} U(t_2)\right) \|_{\dot H^{(\sigma+1)s}}^2 \notag
        \\
        = &\sum_k |k|^{2\sigma s} |k|^{2s} e^{-2\triangle \phi_0 |k|^s}\left| e^{\phi_0(t_1)|k|^s} \hat{U}(t_1,k) -  e^{\phi_0(t_2)|k|^s} \hat{U}(t_2,k) \right|^2 \notag
        \\
        \leq & C_{\triangle \phi_0} \sum_k |k|^{2\sigma s} \left| e^{\phi_0(t_1)|k|^s} \hat{U}(t_1,k) -  e^{\phi_0(t_2)|k|^s} \hat{U}(t_2,k) \right|^2  \notag
        \\
        = &C_{\triangle \phi_0} \|e^{\phi_0(t_1) A^s} U(t_1) - e^{\phi_0(t_2)} U(t_2) \|_{\dot H^{\sigma s}}^2, \label{est:cont}
    \end{align}
    for any $t_1, t_2 \in [0,T]$. Consequently,
    the solution $U\in C\left([0, T]; \mathcal{G}_{\phi_1}^{\sigma+1,s}\right)$. From Proposition \ref{p.061702} and the bound on the initial data \eqref{e.062001}, it follows that $\|U(t)\|_{\mathcal{G}_{\phi_1(t)}^{\sigma,s}}$ is decreasing for $t\in[0,T]$, and 
    \[\|U(t)\|_{\mathcal{G}_{\phi_1(t)}^{\sigma,s}}\leq \|U_0\|_{\mathcal{G}_{\phi_1(0)}^{\sigma,s}}\leq \|U_0\|_{\G_{\alpha+\eta}^{\sigma,s}}.\]
    
    We now iterate this process. Suppose at the $j$-th step we have \eqref{eq:induction} holds. Specifically, it holds that $\|U(jT)\|_{\G_{\phi_j(jT)}^{\sigma,s}} \leq \|U_0\|_{\G_{\alpha+\eta}^{\sigma,s}}$.
    Then, Proposition \ref{p.061701} enables us to extend the solution to $[jT, (j+1)T]$ such that $U\in C\left([0, (j+1)T]; \mathcal{G}_{\phi_j}^{\sigma,s}\right)$. Since $\phi_{j}(t) - \phi_{j+1}(t) =2^{-(j+1)}(\eta-\eta')=:\triangle \phi_j>0$, one can follow the procedure similar to the one used for \eqref{est:cont} to conclude that
    $U\in C\left([0, (j+1)T]; \mathcal{G}_{\phi_{j+1}}^{\sigma+1,s}\right)$. It then follows from Proposition \ref{p.061702} and the initial assumption \eqref{e.062001} that $\|U(t)\|_{\mathcal{G}_{\phi_{j+1}}^{\sigma,s}}$ is monotonically decreasing in $t\in[0,(j+1)T]$ and 
    \[\|U(t)\|_{\mathcal{G}_{\phi_{j+1}(t)}^{\sigma,s}}\leq \|U_0\|_{\mathcal{G}_{\phi_{j+1}(0)}^{\sigma,s}}\leq \|U_0\|_{\G_{\alpha+\eta}^{\sigma,s}},\quad t\in[0,(j+1)T].\]
    This shows that \eqref{eq:induction} holds for the $(j+1)$-th step, thus completing the induction proof of \eqref{eq:induction}.
The construction of $\phi_j$ and the induction process give us the unique solution $U\in C\left([0, \infty); \mathcal{G}_{\phi+\eta'}^{\sigma,s}\right)$ such that 
    $t\to \|U(t)\|_{\mathcal{G}_{\phi(t)+\eta'}^{\sigma,s}}$
    is decreasing for all $t\geq 0$ and $\sup\limits_{t\geq 0} \|U(t)\|_{\mathcal{G}_{\phi(t)+\eta'}^{\sigma,s}} \leq \|U_0\|_{\G_{\alpha+\eta}^{\sigma,s}}$.

    The above argument holds for all $\eta'\in[0,\eta)$, and the sequence of functions $f^{\eta'}(t):= \|U(t)\|_{\mathcal{G}_{\phi+\eta'}^{\sigma,s}}$ is increasing as $\eta' \to \eta$ for each fixed $t\geq 0$, one can apply the monotone convergence theorem to conclude that
    \begin{equation}\label{eq:monotonic}
    \|U(t)\|_{\mathcal{G}_{\phi+\eta}^{\sigma,s}}=\lim_{\eta'\to\eta}\|U(t)\|_{\mathcal{G}_{\phi+\eta'}^{\sigma,s}}\leq \|U_0\|_{\G_{\alpha+\eta}^{\sigma,s}},
    \end{equation}
    which implies that $U\in L^\infty\left(0, \infty; \mathcal{G}_{\phi+\eta}^{\sigma,s}\right)$. To see the continuity in time of $U$ in $\mathcal{G}_{\phi+\eta}^{\sigma,s}$, we first apply Proposition \ref{p.061701} to obtain that the solution $U\in C\left([0, T]; \mathcal{G}_{\phi+\eta}^{\sigma,s}\right)$ since $U_0\in \G_{\alpha+\eta}^{\sigma,s}$. Thanks to \eqref{eq:monotonic}, we know that at the time $t=T$, $\|U(T)\|_{\mathcal{G}_{\phi(T)+\eta}^{\sigma,s}} \leq \|U_0\|_{\G_{\alpha+\eta}^{\sigma,s}}$. As the existence time $T$ depends only on the size of the initial condition and $\phi(t)$ depends linearly on time, we can utilize Proposition \ref{p.061701} again to conclude that $U\in C\left([0, 2T]; \mathcal{G}_{\phi+\eta}^{\sigma,s}\right)$. Repeating this process gives us that $U\in C\left([0, \infty); \mathcal{G}_{\phi+\eta}^{\sigma,s}\right)$. 
    
    Finally, the monotonicity of 
    $t\to \|U(t)\|_{\mathcal{G}_{\phi(t)+\eta}^{\sigma,s}}$
    also follows from the monotone convergence theorem since for $t_1\leq t_2$, 
    \[\|U(t_1)\|_{\mathcal{G}_{\phi(t_1)+\eta}^{\sigma,s}}=\lim_{\eta'\to\eta}\|U(t_1)\|_{\mathcal{G}_{\phi(t_1)+\eta'}^{\sigma,s}}\geq \lim_{\eta'\to\eta}\|U(t_2)\|_{\mathcal{G}_{\phi(t_2)+\eta'}^{\sigma,s}}=\|U(t_2)\|_{\mathcal{G}_{\phi(t_2)+\eta}^{\sigma,s}}.\]
\end{proof}

Finally, we are ready to present the following result regarding the original system \eqref{eq.L061001}: the global well-posedness of $V = \Gamma^{-1}U$ in the Gevrey class of order $\frac1s$. 

\begin{theorem}\label{theorem:global}
   Let $s\in(\frac45,1]$ and $\sigma\in(\frac{8}{5s},2)$, and fix $\eta>0$. For any small $\varepsilon\in (0,1)$, there exists $\alpha=\alpha(\varepsilon)>0$ such that for any initial condition $V_0\in H\cap \G_{\alpha+\eta}^{\sigma,s}$, there exists a unique global solution $V\in C\left([0, \infty) ; \mathcal{G}_{\eta}^{\sigma,s}\right)$ with probability at least $1-\varepsilon$, provided that $\nu$ is sufficiently large.
\end{theorem}
\begin{proof}
    For a given $\varepsilon\in(0,1)$, we choose $\alpha=-4\ln \varepsilon$ so that $e^{-\frac\alpha4}=\varepsilon$. For the initial condition $V_0\in H\cap \G_{\alpha+\eta}^{\sigma,s}$, we take $\nu$ large enough such that $\nu^2 > 4C^* \|V_0\|_{\G_{\alpha+\eta}^{\sigma,s}}$ and $\beta=\frac{\nu^2}4$. Since $U_0=V_0$, these choices of $\nu$ and $\beta$ ensures that the assumption \eqref{e.062001} holds. By Proposition~\ref{thm:main}, we know that system \eqref{eq.L061002} has a unique global solution $U\in C\left([0, \infty) ; \mathcal{G}_{\phi+\eta}^{\sigma,s}\right)$ with $\phi(t)=\alpha+\beta t$ for every sample $\omega \in \Omega_{\alpha, \beta, \nu}$. Now, one can define $V$ by $V = \Gamma^{-1}U$ and follow the proof of Corollary \ref{thm:local} to obtain the unique global solution $V\in C\left([0, \infty) ; \mathcal{G}_{\eta}^{\sigma,s}\right)$ for every sample $\omega \in \Omega_{\alpha, \beta, \nu}$, and thus with probability at least $\mathbb{P}(\Omega_{\alpha, \beta, \nu}) \geq 1-e^{-\frac{\alpha\beta}{\nu^2}}=1-\varepsilon$.
\end{proof}

\section{Random damping}\label{sec:damping}
This section is dedicated to the analysis of the stochastic inviscid PE subject to random damping, i.e., the term $F$ in \eqref{PE-abstract} takes the form $F(V)=\nu V$. Hence, the equation reads 
\begin{align}\label{e.071303}
        d V + \mathcal{P}Q(V,V)dt =  \nu V dW, \qquad V(0)=V_0. 
\end{align}
Recall $\Gamma = e^{-\nu W_t |\nabla|^s}$ and $U = \Gamma V$ as defined in Section~\ref{sec:pre}. Since $s =0$ when considering random damping, $\Gamma$ now takes the form $\Gamma(t) = e^{-\nu W_t}$ and is independent of spatial variables. This gives a random PDE for $U$,
\begin{align}\label{e.L061301}
    \partial_tU + \Gamma^{-1}\mathcal{P}Q(U,U) =  -\frac{1}{2}\nu^2 U, \qquad U(0)=U_0.
\end{align}
In what follows we will focus on this equation. Once we obtain solutions $U$ of \eqref{e.L061301}, we can obtain solutions $V=e^{\nu W}U$ of \eqref{e.071303} with the same regularity as $U$. Since there is no dissipation, the contraction mapping argument used in Section~\ref{sec:diffusion} based on the mild formulation is no longer suitable. Moreover, system \eqref{e.L061301} is still ill-posed in Gevrey class of order strictly greater than $1$ due to the lack of dissipation \cite{renardy2009ill,han2016ill,ibrahim2021finite}. Thus, one has to work in the analytic class, i.e., $\G_{\phi}^{\sigma,s}$ with $s=1$.

We shall solve the random equation \eqref{e.L061301} in a pathwise manner on $[0,T]$ for any fixed $T>0$. For each fixed $\omega\in \Omega$, the sample path $t\to W_t(\omega)$ is continuous, thus  
\[M_{\omega}:=\sup_{t\in[0,T]}\Gamma^{-1}(t)<\infty.\] 
Let the analytic radius $\phi$ be such that $\phi(t)> 0$ and $\phi'(t)\leq 0$ for all $t\in[0,T]$. By using a Galerkin approximation scheme, one can obtain the local existence result as summarized in Proposition~\ref{p.072901} below. To this end, we first recall estimates on the nonlinear term (cf. \cite[Lemma A.1 and A.3]{ghoul2022effect}): for any $\phi\geq 0$ and $\sigma>2$,
\begin{align}\label{e.L061601}
    |\langle e^{\phi A}A^{\sigma}Q(U,U),e^{\phi A}A^{\sigma}U\rangle|_{L^2}\leq C_{\sigma}\|U\|_{\dot \G_{\phi}^{\sigma,1}}\|U\|^2_{\dot \G_{\phi}^{\sigma+\frac12,1}},
\end{align}
for any $U\in H\cap \G_{\phi}^{\sigma+\frac12,1}$.  

\begin{proposition}\label{p.072901}
    Let $\sigma>\frac52$ and $\phi_0>0$ be constants. For any $U_0\in \G_{\phi_0}^{\sigma, 1}$ and $\omega\in\Omega$, there exists a function $\phi\in C^1(\mathbb R)$ (depending on $\omega$) with $\phi(0)=\phi_0$, and a time $\mathcal{T}\leq T$ depending on $\|U_0\|_{\G_{\phi_0}^{\sigma, 1}}$, $\phi_0$, and $M_{\omega}$, such that there is a unique solution $U$ to \eqref{e.L061301} satisfying
    \[U\in C([0,\mathcal{T}];\G_{\phi}^{\sigma,1})\cap L^2(0,\mathcal{T};\G_{\phi}^{\sigma+\frac12,1}).\] 
\end{proposition} 
\begin{proof}
    The proof follows similarly to \cite[Theorem 3.1]{ghoul2022effect}, and we highlight two main differences. First, the additional damping term does not introduce any complications. Second, while \cite[Theorem 3.1]{ghoul2022effect} established that  $U\in L^\infty(0,\mathcal{T};\G_{\phi}^{\sigma,1})\cap L^2(0,\mathcal{T};\G_{\phi}^{\sigma+\frac12,1})$, we will prove here that $U\in C([0,\mathcal{T}];\G_{\phi}^{\sigma,1})$. From Lemma \ref{lemma:a1}, the definition of the nonlinear term $Q(U,U)$, and the fact that $\sigma>\frac52$, one can infer that 
    \begin{align}\label{e.Q071301}
        \|e^{\phi A} A^{\sigma-\frac12} \mathcal P Q(U,U)\|_{L^2} = \sup\limits_{\|\phi\|_{L^2} =1} \langle e^{\phi A} A^{\sigma-\frac12} \mathcal P Q(U,U), \phi\rangle  
       \leq C_{\sigma}\|U\|_{\dot\G_\phi^{\sigma,1}}\|U\|_{\dot\G_\phi^{\sigma+\frac12,1}} .
    \end{align}
    Since $ \partial_t ( e^{\phi A} A^{\sigma} U ) = \dot\phi e^{\phi A} A^{\sigma+1} U + e^{\phi A} A^{\sigma} \partial_t U$, using \eqref{e.L061301} and \eqref{e.Q071301}, we get
    \begin{align*}
       \| \partial_t (e^{\phi A} A^{\sigma} U) \|_{H^{-\frac12}} \leq & C \left(\|e^{\phi A} A^{\sigma+1} U \|_{H^{-\frac12}} +  \|e^{\phi A} A^{\sigma} \mathcal P Q(U,U)\|_{H^{-\frac12}} + \|e^{\phi A} A^{\sigma} U \|_{H^{-\frac12}} \right)
       \\
       \leq & C \left(\|U\|_{\G_{\phi}^{\sigma+\frac12,1}} +  \|U\|_{\G_{\phi}^{\sigma,1}}\|U\|_{\G_{\phi}^{\sigma+\frac12,1}}\right).
    \end{align*}
     This implies that $\partial_t ( e^{\phi A} A^{\sigma} U ) \in L^2(0,\mathcal T; H^{-\frac12})$. As $e^{\phi A} A^{\sigma} U \in L^2(0,\mathcal T; H^{\frac12})$ and $H^{-\frac12}$ is the dual space of $H^{\frac12}$, by the Lions–Magenes lemma, we infer that $e^{\phi A} A^{\sigma} U \in C([0,\mathcal T]; L^2)$, and thus $U\in C([0,\mathcal{T}];\G_{\phi}^{\sigma,1})$.
\end{proof}

We next show that by appropriately choosing the analytic radius and utilizing the damping term, a global solution exists for small initial data with quantifiable high probability. Let $\alpha>0$, $0<\beta<\frac{\nu^2}{2}$, and recall the sample set $\Omega_{\alpha,\beta,\nu} $ defined in \eqref{def:samples}, which has probability 
 $\mathbb{P}(\Omega_{\alpha,\beta,\nu})\geq 1-e^{-\frac{\alpha\beta}{\nu^2}}$. 

\begin{proposition}\label{prop:damping}
    Let $\sigma>\frac52$ and $\phi_0>0$. For the constant $C_\sigma$ appearing in \eqref{e.L061601}, assume that
    \begin{align}\label{e.071301}
        (\nu^2-2\beta)\phi_0>4C_{\sigma}.
    \end{align}
    For any $\omega\in \Omega_{\alpha,\beta,\nu}$ and initial data $U_0$ satisfying 
    \begin{align}\label{e.071302}
        e^{\alpha}\|U_0\|_{\G_{\phi_0}^{\sigma, 1}}\leq \frac{(\nu^2-2\beta)\phi_0}{4C_{\sigma}}-1,
    \end{align}
    there exists a unique global solution $U\in C([0,\infty);\G_{\phi}^{\sigma,1})$ to system \eqref{e.L061301} with analytic radius 
    \begin{align}\label{eq.radius}
        \phi=\phi(t)=\phi_0-\frac{4C_{\sigma}}{\nu^2-2\beta}\left(e^{\alpha}\|V_0\|_{\G_{\phi_0}^{\sigma, 1}}+1\right)\left(1-e^{-(\frac{\nu^2}{2}-\beta)t}\right).
    \end{align}
\end{proposition}

\begin{proof} 
    From equation \eqref{e.L061301}, one has 
    \begin{align*}
        \begin{split}
            \frac12\frac{d}{dt}\|U\|_{\G_{\phi}^{\sigma,1}}^2 + \frac{\nu^2}{2} \|U\|_{\G_{\phi}^{\sigma,1}}^2
            &= \phi'(t)\|U\|_{\dot \G_{\phi}^{\sigma+\frac12,1}}^2+\Gamma^{-1}\langle e^{\phi(t) A}A^{\sigma}Q(U,U),e^{\phi(t) A}A^{\sigma}U\rangle\\
            &\leq \left(\phi'(t)+C_{\sigma}\Gamma^{-1}\|U\|_{\G_{\phi}^{\sigma,1}}\right)\|U\|_{\dot \G_{\phi}^{\sigma+\frac12,1}}^2
        \end{split}
    \end{align*}
    by the nonlinear estimate \eqref{e.L061601}. For $\omega\in \Omega_{\nu,\alpha,\beta}$, we have $ \Gamma^{-1}(t)=e^{\nu W_t}\leq e^{\alpha+\beta t}$ for all $t$. Thus, we derive that
    \begin{align}\label{e.Q071302}
        \begin{split}
            \frac12\frac{d}{dt}\|U\|_{\G_{\phi}^{\sigma,1}}^2 &+ C_\sigma e^{-(\frac{\nu^2}{2}-\beta)t}\|U\|_{
            \dot\G_{\phi}^{\sigma+\frac12,1}}^2+\frac{\nu^2}{2} \|U\|_{\G_{\phi}^{\sigma,1}}^2\\
            &\leq \left(\phi'(t)+C_{\sigma}\Gamma^{-1}\|U\|_{\G_{\phi}^{\sigma,1}}+C_{\sigma} e^{-(\frac{\nu^2}{2}-\beta)t}\right)\|U\|_{
            \dot\G_{\phi}^{\sigma+\frac12,1}}^2\\
            &\leq \left(\phi'(t)+C_{\sigma} e^{-(\frac{\nu^2}{2}-\beta)t}\left( e^{\alpha+\frac{\nu^2}{2}t}\|U\|_{\G_{\phi}^{\sigma,1}}+1\right)\right)\|U\|_{
            \dot\G_{\phi}^{\sigma+\frac12,1}}^2\\
            &=C_{\sigma}e^{-(\frac{\nu^2}{2}-\beta)t}\left(e^{\alpha+\frac{\nu^2}{2}t}\|U\|_{\G_{\phi}^{\sigma,1}}-2e^{\alpha}\|U_0\|_{\G_{\phi_0}^{\sigma, 1}}-1\right)\|U\|_{
            \dot\G_{\phi}^{\sigma+\frac12,1}}^2.
        \end{split}
    \end{align}
    
    Define the stopping time $T^\ast$
    \begin{align*}
        T^* = \min \left\{t\geq 0: e^{\alpha+\frac{\nu^2}{2}t}\|U\|_{\G_{\phi(t)}^{\sigma,1}}-2e^{\alpha}\|U_0\|_{\G_{\phi_0}^{\sigma, 1}}-1 \geq 0 \right\}.
    \end{align*}
    By time continuity, we know that $T^*>0$. We will prove that $T^\ast = \infty$ by contradiction. Assume that $T^*<\infty$, then we know that 
    \begin{align}\label{Q:06191}
        e^{\alpha+\frac{\nu^2}{2}T^*}\|U\|_{\G_{\phi(T^*)}^{\sigma,1}}-2e^{\alpha}\|U_0\|_{\G_{\phi_0}^{\sigma, 1}}-1 = 0,
    \end{align}
    and for $t\in [0, T^*]$ we have
    \[
     \frac12\frac{d}{dt}\|U\|_{\G_{\phi}^{\sigma,1}}^2 + C_\sigma e^{-(\frac{\nu^2}{2}-\beta)t}\|U\|_{\dot\G_{\phi}^{\sigma+\frac12,1}}^2+\frac{\nu^2}{2} \|U\|_{\G_{\phi}^{\sigma,1}}^2 \leq 0.
    \]
    Then one can apply the Gronwall inequality to conclude that
    \[e^{\alpha+\frac{\nu^2}{2}t}\|U\|_{\G_{\phi(t)}^{\sigma,1}}\leq e^{\alpha}\|U_0\|_{\G_{\phi_0}^{\sigma, 1}},\]
    for $t\in[0,T^*]$. In particular, the above holds for $t=T^*$, which leads to a contradiction to \eqref{Q:06191}. Consequently, $T^*=\infty$. Thanks to condition \eqref{e.071302}, one can infer that $\phi(t) >0$ for all $t\geq 0$, thus the analytic solution $U$ exists globally. 
\end{proof}

Finally, the global existence of $V$ with high probability will be established using the construction $V = \Gamma^{-1}U$. This result is summarized below.

\begin{theorem}\label{thm:damping}
    Let $\sigma>\frac52$ and $\phi_0>0$. For any small $\varepsilon\in (0,1)$ and initial condition $V_0\in \G_{\phi_0}^{\sigma,1}$, there exists, with probability at least $1-\varepsilon$, a unique global solution $V\in C([0,\infty);\G_{\phi}^{\sigma,1})$ to system \eqref{e.071303} with analytic radius $\phi(t)$ defined in \eqref{eq.radius}, provided that the damping intensity $\nu$ is sufficiently large.
\end{theorem}
\begin{proof}
    For any $\varepsilon\in(0,1)$ and $V_0\in \G_{\phi_0}^{\sigma,1}$, set $\alpha=-4\ln \varepsilon$ so that $e^{-\frac\alpha4}=\varepsilon$. Choosing $\beta=\frac{\nu^2}4$ and $\nu$ sufficiently large such that such that 
    \begin{align*}
        \nu^2 \geq \frac{8C_\sigma}{\phi_0}(e^{\alpha}\|V_0\|_{\G_{\phi_0}^{\sigma,1}} +1) = \frac{8C_\sigma}{\phi_0}(\varepsilon^{-4}\|V_0\|_{\G_{\phi_0}^{\sigma,1}} +1),
    \end{align*}
    then both conditions \eqref{e.071301} and \eqref{e.071302} are satisfied. Since $U_0=V_0$, Proposition~\ref{prop:damping} ensures that there exists a unique global solution $U\in C([0,\infty);\G_{\phi}^{\sigma,1})$ to system \eqref{e.L061301} with analytic radius $\phi(t)$ defined in \eqref{eq.radius} for any $\omega \in \Omega_{\alpha,\beta,\nu}$. Consequently, using $V=e^{\nu W_t} U$, one obtains a unique global solution $V\in C([0,\infty);\G_{\phi}^{\sigma,1})$ to system \eqref{e.071303} for any $\omega \in \Omega_{\alpha,\beta,\nu}$, with probability at least $1-e^{-\frac{\alpha\beta}{\nu^2}} = 1-\varepsilon$.
\end{proof}

\begin{remark}
    Compared to the random diffusion case (Theorem \ref{theorem:global}), here we do not need to assume the initial analytic radius to be large. 
\end{remark}
\begin{remark}
    By taking $R=e^{\alpha}$, the event $\Omega_{\alpha,\beta,\nu}$ corresponds to the assertion that the geometric Brownian motion $\exp(\nu W_t-\beta t)$ never exceeds $R$, as in \cite{buckmaster2020surface}. This probability can be made arbitrarily close to 1 by choosing a sufficiently large $\alpha$. 
\end{remark}

\section{Conclusion}\label{sec:conclusion}
It is well-known that the deterministic inviscid PE is ill-posed in Sobolev spaces and Gevrey class of order strictly greater than $1$, and some analytic solutions can form singularities in finite time. Therefore, only local well-posedness in the analytic class is sought. In this work, we investigated the regularization effect of noise on the inviscid PE. 
We prove that appropriate random diffusion can restore local well-posedness in certain Gevrey classes with probability one. Moreover, global solutions can be obtained with high probability when the diffusion intensity and the Gevrey radius for the initial condition are sufficiently large. Additionaly, we establish that random damping can regularize the system, meaning that analytic solutions exist globally in time with high probability, provided that the damping intensity is sufficiently large.

It is interesting to observe how different types of random noise can regularize the system in various ways.
Random diffusion allows us to obtain solutions with a non-decreasing Gevrey radius. However, to achieve global solutions with high probability, we need to impose a sufficiently smooth condition on the initial data, specifically by ensuring it has a large enough Gevrey radius. On the other hand, in the case of random damping, we exploit the dissipation from a decreasing analytic radius to handle the nonlinearity, leading to solutions with a decreasing analytic radius. Nonetheless, a large damping intensity can slow the rate at which the analytic radius decreases, thereby enabling global solutions with high probability. 

It is intriguing to explore whether appropriate types of noise could restore local well-posedness for the inviscid PE in Sobolev spaces and Gevrey class of order $\leq \frac45$. We leave this investigation as a topic for future work.

\section*{Acknowledgments}
Q.L. was partially supported by an AMS-Simons Travel Grant. R.H. was partially supported by a grant from the Simons Foundation (MP-TSM-00002783).

\appendix
\section{Nonlinear estimates}
\begin{lemma}\label{lemma:a1}
    Given smooth periodic functions $f$, $g$, and $h$ such that $f$ and $g$ have zero means over $\mathbb T^3$, the following inequality holds for any $r\geq 0$, $\phi\geq 0$, and $\eta>0$:
    \begin{equation*}\label{lemma-type1-inequality}
    \begin{split}
        \left|\left\langle e^{\phi A} A^r  (fg),   h  \right\rangle\right| \leq  C_r\left( \|e^{\phi A} A^{r}  f\| \|e^{\phi A} A^{\frac32+\eta}  g\|  + \|e^{\phi A} A^{\frac32+\eta}  f\| \|e^{\phi A} A^{r}  g\| \right) \|  h\|,
    \end{split}
\end{equation*}
where $C_r$ is a constant depending on $r$.
\end{lemma}
\begin{proof}
    We start with writing the Fourier representations of $f, g$ and $h$:
    \begin{eqnarray*}
    &&\hskip-.8in
     f(x) = \sum\limits_{j\in  2\pi\mathbb{Z}^3} \hat{f}_{j} e^{ ij\cdot x}, \quad
    g(x) = \sum\limits_{k\in 2\pi\mathbb{Z}^3} \hat{g}_{k} e^{ ik\cdot x}, \quad
    h(x) = \sum\limits_{l\in 2\pi\mathbb{Z}^3} \hat{h}_{l} e^{ il\cdot x}. 
    \end{eqnarray*}
    Then 
    \begin{eqnarray*}
    \left|\left\langle e^{\phi A} A^r  (fg),  h  \right\rangle\right| = \left|\left\langle fg,e^{\phi A} A^r  h  \right\rangle\right| \leq  \sum\limits_{j+k+l=0} |\hat{f}_{j}||\hat{g}_{k}||l|^{r} e^{\phi |l|} |\hat{h}_{l}|.
    \end{eqnarray*}
Since $|l| = |j+k| \leq |j|+|k|$, we have 
$
    |l|^{r} \leq (|j|+|k|)^{r} \leq C_r(|j|^{r} + |k|^{r})
$
and 
$e^{\phi |l|} \leq e^{\phi |j|} e^{\phi |k|}$,
thus
\begin{eqnarray*}
\left|\left\langle A^r  (fg),  h  \right\rangle\right| \leq  \sum\limits_{j+k+l=0} C_r (|j|^{r}+|k|^{r})e^{\phi |j|}|\hat{f}_{j}| e^{\phi |k|}|\hat{g}_{k}||\hat{h}_{l}| := A_1 + A_2.
\end{eqnarray*}
Regarding the term $A_1$, with $g$ having zero mean, using  the Cauchy–Schwarz inequality, we deduce 
\begin{eqnarray*}
A_1&&= \sum\limits_{j+k+l=0} C_r |j|^{r}  e^{\phi |j|} |\hat{f}_{j}| e^{\phi |k|}|\hat{g}_{k}| |\hat{h}_{l}| 
= C_r \sum\limits_{\substack{k\in 2\pi \mathbb{Z}^3 \\ k\neq 0} }  e^{\phi |k|}|\hat{g}_{k}|  \sum\limits_{\substack{j\in 2\pi \mathbb{Z}^3 \\ j\neq 0} } |j|^{r} e^{\phi |j|} |\hat{f}_{j}|  |\hat{h}_{-j-k}| \nonumber\\
&&\leq C_r \Big( \sum\limits_{\substack{k\in 2\pi \mathbb{Z}^3 \\ k\neq 0} } |k|^{-3-2\eta}\Big)^{\frac{1}{2}} \Big( \sum\limits_{\substack{k\in 2\pi \mathbb{Z}^3 \\ k\neq 0} } |k|^{3+2\eta}  e^{2\phi |k|} |\hat{g}_{k}|^2\Big)^{\frac{1}{2}}\\
&&\hskip1in \times\sup\limits_{k\in 2\pi \mathbb{Z}^3}\Big( \sum\limits_{\substack{j\in 2\pi \mathbb{Z}^3 \\ j\neq 0} } |j|^{2r} e^{2\phi |j|} |\hat{f}_{j}|^2\Big)^{\frac{1}{2}} \Big( \sum\limits_{\substack{j\in 2\pi \mathbb{Z}^3 \\ j\neq 0} } |\hat{h}_{-j-k}|^2\Big)^{\frac{1}{2}} \nonumber\\
&& \leq C_r \|e^{\phi A} A^{r}  f\| \|e^{\phi A} A^{\frac32+\eta}  g\| \|  h\|,
\end{eqnarray*}
for any $\eta>0$. 

For $A_2$, since $f$ has zero mean, similarly we have
\begin{eqnarray*}
&&\hskip-.8in 
A_2 = \sum\limits_{j+k+l=0} C_r |k|^{2r} e^{\phi |j|}|\hat{f}_{j}| e^{\phi |k|}|\hat{g}_{k}| |\hat{h}_{l}| 
\leq C_r \|e^{\phi A} A^{\frac32+\eta}  f\| \|e^{\phi A} A^{r}  g\| \|h\|.
\end{eqnarray*}
\end{proof}

\bibliographystyle{plain}
\bibliography{Reference}

\begin{thebibliography}{10}

\bibitem{adams2003sobolev}
Robert~A Adams and John~JF Fournier.
\newblock {\em Sobolev spaces}.
\newblock Elsevier, 2003.

\bibitem{agresti2022stochastic-2}
Antonio Agresti, Matthias Hieber, Amru Hussein, and Martin Saal.
\newblock The stochastic primitive equations with non-isothermal turbulent pressure.
\newblock {\em arXiv preprint arXiv:2210.05973}, 2022.

\bibitem{agresti2022stochastic-1}
Antonio Agresti, Matthias Hieber, Amru Hussein, and Martin Saal.
\newblock The stochastic primitive equations with transport noise and turbulent pressure.
\newblock {\em Stochastics and Partial Differential Equations: Analysis and Computations}, pages 1--81, 2022.

\bibitem{azerad2001mathematical}
Pascal Az{\'e}rad and Francisco Guill{\'e}n.
\newblock Mathematical justification of the hydrostatic approximation in the primitive equations of geophysical fluid dynamics.
\newblock {\em SIAM Journal on Mathematical Analysis}, 33(4):847--859, 2001.

\bibitem{brenier1999homogeneous}
Yann Brenier.
\newblock Homogeneous hydrostatic flows with convex velocity profiles.
\newblock {\em Nonlinearity}, 12(3):495, 1999.

\bibitem{brenier2003remarks}
Yann Brenier.
\newblock Remarks on the derivation of the hydrostatic {Euler} equations.
\newblock {\em Bulletin des Sciences Mathematiques}, 127(7):585--595, 2003.

\bibitem{brzezniak2021well}
Zdzislaw Brze{\'z}niak and Jakub Slav{\'\i}k.
\newblock Well-posedness of the {3D} stochastic primitive equations with multiplicative and transport noise.
\newblock {\em Journal of Differential Equations}, 296:617--676, 2021.

\bibitem{buckmaster2020surface}
Tristan Buckmaster, Andrea Nahmod, Gigliola Staffilani, and Klaus Widmayer.
\newblock The surface quasi-geostrophic equation with random diffusion.
\newblock {\em International Mathematics Research Notices}, 2020(23):9370--9385, 2020.

\bibitem{cao2015finite}
Chongsheng Cao, Slim Ibrahim, Kenji Nakanishi, and Edriss~S Titi.
\newblock Finite-time blowup for the inviscid primitive equations of oceanic and atmospheric dynamics.
\newblock {\em Communications in Mathematical Physics}, 337(2):473--482, 2015.

\bibitem{cao2007global}
Chongsheng Cao and Edriss~S Titi.
\newblock Global well-posedness of the three-dimensional viscous primitive equations of large scale ocean and atmosphere dynamics.
\newblock {\em Annals of Mathematics}, pages 245--267, 2007.

\bibitem{chen2022stable}
Jiajie Chen and Thomas~Y Hou.
\newblock Stable nearly self-similar blowup of the {2D Boussinesq and 3D Euler} equations with smooth data {I}: Analysis.
\newblock {\em arXiv preprint arXiv:2210.07191}, 2022.

\bibitem{collot2023stable}
Charles Collot, Slim Ibrahim, and Quyuan Lin.
\newblock Stable singularity formation for the inviscid primitive equations.
\newblock {\em Annales de l'Institut Henri Poincar{\'e} C}, 2023.

\bibitem{debussche2011local}
Arnaud Debussche, Nathan Glatt-Holtz, and Roger Temam.
\newblock Local martingale and pathwise solutions for an abstract fluids model.
\newblock {\em Physica D: Nonlinear Phenomena}, 240(14-15):1123--1144, 2011.

\bibitem{debussche2012global}
Arnaud Debussche, Nathan Glatt-Holtz, Roger Temam, and Mohammed Ziane.
\newblock Global existence and regularity for the {3D} stochastic primitive equations of the ocean and atmosphere with multiplicative white noise.
\newblock {\em Nonlinearity}, 25(7):2093, 2012.

\bibitem{elgindi2021finite}
Tarek~M Elgindi.
\newblock Finite-time singularity formation for {$ C^{1,\alpha} $} solutions to the incompressible {Euler} equations on {$\mathbb {R}^ 3$}.
\newblock {\em Annals of Mathematics}, 194(3):647--727, 2021.

\bibitem{furukawa2020rigorous}
Ken Furukawa, Yoshikazu Giga, Matthias Hieber, Amru Hussein, Takahito Kashiwabara, and Marc Wrona.
\newblock Rigorous justification of the hydrostatic approximation for the primitive equations by scaled {N}avier--{S}tokes equations.
\newblock {\em Nonlinearity}, 33(12):6502, 2020.

\bibitem{ghoul2022effect}
Tej~Eddine Ghoul, Slim Ibrahim, Quyuan Lin, and Edriss~S Titi.
\newblock On the effect of rotation on the life-span of analytic solutions to the {3D} inviscid primitive equations.
\newblock {\em Archive for rational mechanics and analysis}, 243(2):747--806, 2022.

\bibitem{glatt2014existence}
Nathan Glatt-Holtz, Igor Kukavica, Vlad Vicol, and Mohammed Ziane.
\newblock Existence and regularity of invariant measures for the three dimensional stochastic primitive equations.
\newblock {\em Journal of Mathematical Physics}, 55(5):051504, 2014.

\bibitem{glatt2011pathwise}
Nathan Glatt-Holtz and Roger Temam.
\newblock Pathwise solutions of the {2D} stochastic primitive equations.
\newblock {\em Applied Mathematics \& Optimization}, 63(3):401--433, 2011.

\bibitem{glatt2008stochastic}
Nathan Glatt-Holtz and Mohammed Ziane.
\newblock The stochastic primitive equations in two space dimensions with multiplicative noise.
\newblock {\em Discrete \& Continuous Dynamical Systems-B}, 10(4):801, 2008.

\bibitem{glatt2014local}
Nathan~E Glatt-Holtz and Vlad~C Vicol.
\newblock Local and global existence of smooth solutions for the stochastic {Euler} equations with multiplicative noise.
\newblock {\em The Annals of Probability}, 42(1):80--145, 2014.

\bibitem{grenier1999derivation}
Emmanuel Grenier.
\newblock On the derivation of homogeneous hydrostatic equations.
\newblock {\em ESAIM: Mathematical Modelling and Numerical Analysis}, 33(5):965--970, 1999.

\bibitem{han2016ill}
Daniel Han-Kwan and Toan~T Nguyen.
\newblock Ill-posedness of the hydrostatic {Euler} and singular {Vlasov} equations.
\newblock {\em Archive for Rational Mechanics and Analysis}, 221(3):1317--1344, 2016.

\bibitem{hieber2016global}
Matthias Hieber and Takahito Kashiwabara.
\newblock Global strong well-posedness of the three dimensional primitive equations in {$L^p$}-spaces.
\newblock {\em Archive for Rational Mechanics and Analysis}, 221(3):1077--1115, 2016.

\bibitem{hu2023local}
Ruimeng Hu and Quyuan Lin.
\newblock Local martingale solutions and pathwise uniqueness for the three-dimensional stochastic inviscid primitive equations.
\newblock {\em Stochastics and Partial Differential Equations: Analysis and Computations}, 11(4):1470--1518, 2023.

\bibitem{hu2023pathwise}
Ruimeng Hu and Quyuan Lin.
\newblock Pathwise solutions for stochastic hydrostatic {E}uler equations and hydrostatic {N}avier-{S}tokes equations under the local {R}ayleigh condition.
\newblock {\em arXiv preprint arXiv:2301.07810}, 2023.

\bibitem{ibrahim2021finite}
Slim Ibrahim, Quyuan Lin, and Edriss~S Titi.
\newblock Finite-time blowup and ill-posedness in sobolev spaces of the inviscid primitive equations with rotation.
\newblock {\em Journal of Differential Equations}, 286:557--577, 2021.

\bibitem{kobelkov2006existence}
Georgij~M Kobelkov.
\newblock Existence of a solution ‘in the large’ for the {3D} large-scale ocean dynamics equations.
\newblock {\em Comptes Rendus Mathematique}, 343(4):283--286, 2006.

\bibitem{kukavica2014local}
Igor Kukavica, Nader Masmoudi, Vlad Vicol, and Tak~Kwong Wong.
\newblock On the local well-posedness of the {Prandtl and hydrostatic Euler} equations with multiple monotonicity regions.
\newblock {\em SIAM Journal on Mathematical Analysis}, 46(6):3865--3890, 2014.

\bibitem{kukavica2011local}
Igor Kukavica, Roger Temam, Vlad~C Vicol, and Mohammed Ziane.
\newblock Local existence and uniqueness for the hydrostatic {Euler} equations on a bounded domain.
\newblock {\em Journal of Differential Equations}, 250(3):1719--1746, 2011.

\bibitem{kukavica2007regularity}
Igor Kukavica and Mohammed Ziane.
\newblock On the regularity of the primitive equations of the ocean.
\newblock {\em Nonlinearity}, 20(12):2739, 2007.

\bibitem{li2019primitive}
Jinkai Li and Edriss~S Titi.
\newblock The primitive equations as the small aspect ratio limit of the {Navier--Stokes} equations: Rigorous justification of the hydrostatic approximation.
\newblock {\em Journal de Math{\'e}matiques Pures et Appliqu{\'e}es}, 124:30--58, 2019.

\bibitem{li2022primitive}
Jinkai Li, Edriss~S Titi, and Guozhi Yuan.
\newblock The primitive equations approximation of the anisotropic horizontally viscous {3D Navier-Stokes} equations.
\newblock {\em Journal of Differential Equations}, 306:492--524, 2022.

\bibitem{masmoudi2012h}
Nader Masmoudi and Tak~Kwong Wong.
\newblock On the {$H^s$} theory of hydrostatic {Euler} equations.
\newblock {\em Archive for Rational Mechanics and Analysis}, 204(1):231--271, 2012.

\bibitem{renardy2009ill}
Michael Renardy.
\newblock Ill-posedness of the hydrostatic {Euler and Navier-Stokes} equations.
\newblock {\em Archive for rational mechanics and analysis}, 194(3):877--886, 2009.

\bibitem{resnick2013adventures}
Sidney~I Resnick.
\newblock {\em Adventures in stochastic processes}.
\newblock Springer Science \& Business Media, 2013.

\bibitem{rosenzweig2023global}
Matthew Rosenzweig and Gigliola Staffilani.
\newblock Global solutions of aggregation equations and other flows with random diffusion.
\newblock {\em Probability Theory and Related Fields}, 185(3):1219--1262, 2023.

\bibitem{wong2015blowup}
Tak~Kwong Wong.
\newblock Blowup of solutions of the hydrostatic {Euler} equations.
\newblock {\em Proceedings of the American Mathematical Society}, 143(3):1119--1125, 2015.

\end{thebibliography}
\end{document}